\documentclass[letterpaper,conference]{ieeeconf}
\IEEEoverridecommandlockouts                              
\usepackage[utf8]{inputenc}
\usepackage{textcomp}

\usepackage{graphicx}
\usepackage{amsmath}
\usepackage[version=4]{mhchem}
\usepackage{siunitx}
\usepackage{graphicx}
\usepackage{longtable,tabularx}

\let\OriginalLabelIndent\labelindent
\let\labelindent\relax

\usepackage{amsmath} 
\usepackage{amsfonts}
\usepackage{amsthm}

\theoremstyle{plain}
\newtheorem{theorem}{Theorem}[section]
\newtheorem{proposition}{Proposition}[section]
\newtheorem{corollary}{Corollary}[section]

\theoremstyle{definition}
\newtheorem{assumption}{Assumption}[section]

\newtheorem{problem}{Problem}[section]

\newtheorem{remark}{Remark}[section]

\usepackage{float} 
\usepackage[skip=0em,font=footnotesize]{caption}

\usepackage[margin = 1in]{geometry}

\usepackage{tikz}
\usetikzlibrary{shapes,arrows,calc,positioning}
\tikzstyle{bigblock} = [draw, fill=blue!20, rectangle, 
    minimum height=6em, minimum width=8em]
\tikzstyle{medblock} = [draw, fill=blue!20, rectangle, 
    minimum height=4em, minimum width=4em]    
\tikzstyle{mux} = [draw, fill=black!20, rectangle, 
    minimum height=5em, minimum width=0.1em]    
\tikzstyle{smallblock} = [draw, fill=blue!20, rectangle, 
    minimum height=2em, minimum width=3em]
    
\tikzstyle{data_block} = [draw, fill=green!20, rectangle, 
    minimum height=2em, minimum width=3em]
\tikzstyle{ops_block} = [draw, fill=blue!20, rectangle, 
    minimum height=2em, minimum width=3em]    
\tikzstyle{est_block} = [draw, fill=red!20, rectangle, 
    minimum height=2em, minimum width=3em]    
    
\tikzstyle{sum} = [draw, fill=blue!20, circle, node distance=1cm,minimum height=0.5cm]
\tikzstyle{signal} = [coordinate]
\tikzstyle{pinstyle} = [pin edge={to-,thin,black}]
\tikzstyle{block} = [draw, fill=blue!20, rectangle, 
    minimum height=3em, minimum width=9em]
\tikzstyle{blockS} = [draw, fill=blue!20, rectangle, 
    minimum height=3em, minimum width=4em]    
\tikzstyle{input} = [coordinate]
\tikzstyle{output} = [coordinate]
\usetikzlibrary{matrix}

\usetikzlibrary{positioning}
\usetikzlibrary{math} 

\usetikzlibrary{arrows.meta}

\usepackage{hyperref}
\usepackage{xcolor}
\hypersetup{
    colorlinks,
    linkcolor={blue!100!black},
    citecolor={blue!50!black},
    urlcolor={blue!80!black}
}

\newcommand{\bc}{\begin{center}}
\newcommand{\ec}{\end{center}}
\newcommand{\benum}{\begin{enumerate}}
\newcommand{\eenum}{\end{enumerate}}

\newcommand{\matl}{\left[ \begin{array}}
\newcommand{\matr}{\end{array} \right]}

\renewcommand{\matl}{\begin{bmatrix}}
\renewcommand{\matr}{\end{bmatrix}}

\newcommand{\matls}{\left[ \begin{smallmatrix}}
\newcommand{\matrs}{\end{smallmatrix} \right]}
\newcommand{\isdef}{\stackrel{\triangle}{=}}

\newcommand{\rmT}{{\rm T}}

\newcommand{\rmc}{{\rm c}}
\newcommand{\rmd}{{\rm d}}

\newcommand{\rmf}{{\rm f}}

\newcommand{\rmp}{{\rm p}}

\newcommand{\rmv}{{\rm v}}

\newcommand{\BBR}{{\mathbb R}}

\newcommand{\SP}{{\mathcal P}}

\newcommand{\ST}{{\mathcal T}}

\usepackage{enumitem,amssymb}
\newlist{todolist}{itemize}{2}
\setlist[todolist]{label=$\square$}
\usepackage{pifont}

\title{Finite-Time Safe-Set Recovery for Systems with Parametric Uncertainty}

\title{Reach-Stabilize Control of Control-Affine Systems \\ with Unknown Affine Parameters}

\author{
Alexander Dorsey \and Kunal Garg
\and 
Ankit Goel
\thanks{Alexander Dorsey and Ankit Goel are with the Department of Mechanical Engineering, University of Maryland, Baltimore County,1000 Hilltop Circle, Baltimore, MD 21250. {\tt\small adorsey4@umbc.edu, ankgoel@umbc.edu }}%
\thanks{Kunal Garg is with the School for Engineering of Matter, Transport, and Energy at Arizona State University, Tempe, AZ, 85281. Email: \texttt{kgarg24@asu.edu}.}
}

\begin{document}

\maketitle

\begin{abstract}
This paper considers the reach-stabilize problem for a class of nonlinear control-affine systems with unknown parametric uncertainties, where the system states must remain in a safe set at all times, or enter a safe set in finite time and then remain in it, and converge to a goal point.
An estimator is designed that generates a parameter estimate and a computable, nonincreasing bound on the estimation error from a known initial error bound, without requiring persistence of excitation.
The bound defines margins that are added to a control barrier and a control Lyapunov function condition, which are then enforced in a quadratic program for efficient control design.
It is shown that, for the closed-loop system, the safe set is forward invariant when the system starts within the safe set and is finite-time reachable from outside the safe set with a recovery-time bound that does not depend on the unknown parameter, and that the goal point is exponentially stable.
In two numerical examples, with the initial state outside and inside the safe set, the proposed control policy, a baseline \textit{oracle} control policy that uses the true parameter, both recover or remain in the safe set and converge to the goal point, while another baseline control policy that uses a fixed, incorrect, parameter does not remain in the safe set and does not converge to the goal point.
\end{abstract}

\section{Introduction}

Safety-critical systems must converge to a goal while remaining safe.
The \textit{reach-avoid problem} requires the state to reach a target set while avoiding an unsafe region \cite{fisac2015reachavoid,margellos2011hamiltonjacobi}.
The \textit{stabilize-avoid problem} \cite{so2023solving} requires the state to converge to the goal and remain near it while avoiding the unsafe set, and is formulated for initial states in the safe set.
The \textit{reach-stabilize problem} considered here allows the initial state to lie outside the safe set.
The state must reach the safe set in finite time, remain in the set thereafter, and converge to a goal point.
In reach-avoid and stabilize-avoid, reach refers to the goal.
In reach-stabilize, reach refers to the safe set and stabilize refers to the goal point.

Control barrier functions (CBFs) and control Lyapunov functions (CLFs) provide systematic tools for certifying safety and convergence guarantees, respectively, for known system models \cite{ames2017control,romdlony2014uniting,ames2019control}.
In practice, the model parameters are uncertain, and a controller certified under an assumed model need not retain either guarantee under the true parameters \cite{xu2015robustness,kolathaya2018inputtostate}.
The safety guarantee of a standard CBF in terms of the forward invariance of a safe set only holds if the state starts inside the set.
The state can be outside the set at the start of operation or at a later time.
A violation can arise from an event acting on the system, such as an external impulse or an actuator interruption, or from an abrupt change of the safe set that does not involve the system's state.
For instance, when the coverage cone of a ground station's radar is repositioned, the region that is safe for an aircraft or a mobile ground agent changes instantly, and a state that was safe a moment earlier becomes a violation that the agent must correct.
The recovery must be certified when the parameters are unknown, and the uncertainty must be compensated while the system recovers, not only after it has returned to the safe set.
A controller is therefore needed that solves the reach-stabilize problem under unknown parameters with a formal guarantee.
The three problems are shown in Figure \ref{fig:problems}.

\begin{figure}[t]
\centering
\resizebox{\columnwidth}{!}{%
\begin{tikzpicture}[
    >=Stealth,
    traj/.style={very thick,-{Stealth[length=2.2mm]}},
    note/.style={anchor=west,text width=2.5cm,align=left,font=\footnotesize},
    lead/.style={thin}
]
\colorlet{cra}{blue!70!black}
\colorlet{csa}{orange!85!black}
\colorlet{crs}{green!45!black}

\draw[thick] (2,1.5) ellipse (1.9 and 1.35);
\node at (1.0,2.5) {$\mathcal C$};
\node[font=\footnotesize,gray] at (3.85,-0.05) {unsafe};

\draw[fill=gray!25] (2.3,1.6) circle (0.43);
\fill (2.3,1.6) circle (1.4pt);
\node[font=\footnotesize] at (2.5,1.76) {$x_{\mathrm c}$};
\node[font=\footnotesize] at (1.9,1.1) {$\mathcal T$};

\fill[cra] (0.5,2.05) circle (1.6pt);
\draw[traj,cra] (0.5,2.05) .. controls (1.2,2.5) and (1.9,2.3) .. (2.14,1.88);
\draw[traj,cra,dashed] (2.14,1.88) .. controls (2.6,1.95) and (3.0,1.9) .. (3.3,1.4);
\node[note,text=cra] (na) at (4.3,2.5)
  {reach-avoid: reaches $\mathcal T$, need not remain in it};
\draw[lead,cra] (na.west) -- (2.3,2.17);

\fill[csa] (3.1,0.75) circle (1.6pt);
\draw[traj,csa] (3.1,0.75) .. controls (3.2,1.3) and (2.7,1.3) .. (2.34,1.57);
\node[note,text=csa] (ns) at (4.3,0.55)
  {stabilize-avoid: starts in $\mathcal C$, converges to $x_{\mathrm c}$};
\draw[lead,csa] (ns.west) -- (3.1,0.75);

\fill[crs] (-0.3,0.0) circle (1.6pt);
\draw[very thick,crs] (-0.3,0.0) .. controls (0.0,0.3) and (0.35,0.5) .. (0.55,0.63);
\node[crs,font=\LARGE] at (0.55,0.63) {$\star$};
\node[anchor=east,font=\scriptsize,crs] at (0.42,0.66) {entry};
\draw[traj,crs] (0.55,0.63) .. controls (1.1,1.0) and (1.6,1.2) .. (2.26,1.57);
\node[note,text=crs,text width=4.6cm] (nr) at (-0.5,-0.85)
  {reach-stabilize: starts outside $\mathcal C$, enters $\mathcal C$ in finite time, converges to $x_{\mathrm c}$};
\draw[lead,crs] (-0.3,-0.55) -- (-0.3,0.0);
\end{tikzpicture}}
\caption{The three problems for a safe set $\mathcal C$, target set $\mathcal T$, and goal $x_{\mathrm c}$.
Reach-avoid requires the state to reach $\mathcal T$ while avoiding the unsafe region.
Stabilize-avoid requires convergence to the goal from an initial state in $\mathcal C$.
Reach-stabilize allows an initial state outside $\mathcal C$, and requires entry into $\mathcal C$ in finite time, invariance thereafter, and convergence to the goal.}
\label{fig:problems}
\end{figure}
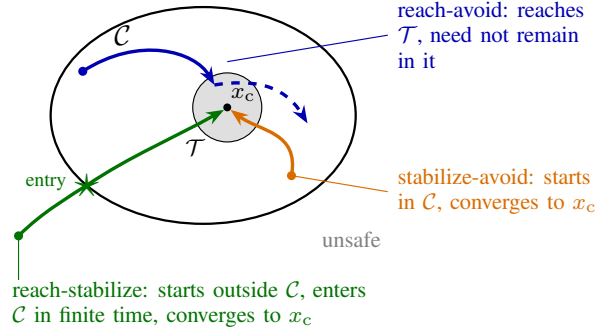

Recovery from an initial violation is addressed for known models by finite-time convergence CBFs, which certify reachability of a superlevel set from a violating state and forward invariance after entry \cite{srinivasan2018finite}.
Robust finite-time CBFs extend the recovery guarantee to bounded disturbances or adversarial inputs, but retain a fixed worst-case uncertainty description \cite{hassan2024resilience}.
The certified property in these works is reaching; remaining in a set, and convergence to a goal point is not part of it.
Robust CBFs combined with fixed-time CLFs enforce forward invariance of a safe set and convergence to a goal set within a user-defined time under additive disturbances and state-estimation errors \cite{garg2021robust,garg2022fixedtime}.
These works treat uncertainty as a bounded disturbance, use a conservative method based on worst-case bounds, and do not utilize online parameter estimation to reduce conservatism.

Stabilize-avoid under parametric uncertainty is addressed by adaptive CBFs and adaptive CLFs, assuming the system starts inside the safe set \cite{taylor2020adaptive,isaly2021adaptive,lopez2023unmatched}.
Adaptive CBFs are introduced in \cite{taylor2020adaptive} and unified with adaptive CLFs in a single QP.
The safety guarantee there is forward invariance of a family of sets parameterized by the current estimate.
The update laws for the CLF and the CBF need not be simultaneously satisfiable under one estimate, and the combined controller thus maintains two separate estimates, one adapted for stabilization and one for safety.
Robust adaptive CBFs combine a tightened safe set with set-membership identification of the parameters \cite{lopez2021robust}.
Stabilization to a desired terminal state is handled by a separate controller with its own parameter estimate.

The closest related work is \cite{black2021fixedtime}, where a fixed-time adaptation law for parameter-affine uncertainty is developed, which requires a persistently excited regressor. In that work, a decreasing bound on the estimation error enters a CBF condition that certifies forward invariance of a tightened safe set from an initial state inside it.
The convergence to a goal point is enforced through a CLF constraint relaxed by a slack variable and hence, is not guaranteed. Moreover, recovery from an initial state outside the safe set is not addressed.
On the other hand, the proposed method here certifies both without persistence of excitation, under a single estimate and a single error bound.


Recent methods establish nonincreasing or vanishing CBF conservatism under excitation \cite{gutierrez2025vanishing}, or enlarge the admissible control set relative to a robust CBF \cite{kamaldar2026composite}.
Across these works, the certified property under parametric uncertainty is forward invariance from an initially safe state.
Convergence to a goal point is either certified through a second, separately adapted mechanism \cite{taylor2020adaptive,lopez2021robust}, or evaluated in simulation without a closed-loop guarantee \cite{black2021fixedtime}.
To the best of our knowledge, the reach-stabilize problem under unknown constant parameters has not been addressed, that is, finite-time recovery from an initial violation, forward invariance thereafter, and convergence to a goal point, certified under a single online estimate in a single QP.

This paper solves the reach-stabilize problem for control-affine systems in which the uncertainty enters as a known regressor matrix multiplying an unknown constant parameter vector.
The solution is a controller in which one online estimate of the parameter and one computable bound on its error jointly certify recovery of the safe set, forward invariance, and convergence to the goal point.
The contributions of this paper are as follows.
\begin{enumerate}
    \item The paper presents a novel mechanism to certify reach-stabilize behavior under parametric uncertainty.
    From an initial state inside the safe set, the state remains in the set.
    From an initial state outside the set, the state enters the set in finite time and remains there.
    In both cases, the state converges to a goal point.
    \item The certified finite-time recovery bound equals that of the known-parameter design and does not depend on the unknown parameter.
    The uncertainty thus changes the control effort required to meet the bound, and does not change the bound.
    \item The guarantees require a known bound on the initial estimation error and do not require persistence of excitation.
    Excitation over an interval tightens the error bound and thus reduces the margins, and the guarantees hold when no such excitation occurs.
\end{enumerate}

The paper is organized as follows.
Section \ref{sec:problem} formulates the problem.
Section \ref{sec:para_estimator} constructs the estimate and the error bound.
Section \ref{sec:control} designs the modified CBF and CLF conditions and the QP, and proves the guarantees.
Section \ref{sec:numerical} presents numerical examples, and Section \ref{sec:conclusions} concludes the paper.

\section{Problem Formulation}
\label{sec:problem}
Consider the control-affine system
\begin{equation}
    \dot x=f(x)+g(x)u+\phi(x)\theta^\star,
    \label{eq:uncertain-system}
\end{equation}
where $x\in\mathcal D\subseteq\BBR^n$ is the measured state, $u\in\mathcal U\subset\BBR^m$ is the input, 
$f\colon\mathcal D\to\BBR^n,$
$g\colon\mathcal D\to\BBR^{n\times m},$ and
$\phi\colon\mathcal D\to\BBR^{n\times p}$ are known and locally Lipschitz, and 
$\theta^\star\in\BBR^p$ is an unknown parameter vector.
%
%
The input set is the nonempty compact polytope $\mathcal U=\{u\in\BBR^m\colon Pu\leq q\}.$
Let
\begin{equation*}
    \mathcal C \isdef \{x\in\mathcal D:h(x)\geq0\},
\end{equation*}
where $h\in C^1$ has relative degree one on $\mathcal D.$
Let $x_\rmc\in\mathcal D$ denote the goal point.
The initial state $x(0)$ may lie outside $\mathcal C,$ for instance following an abrupt event such as an external impulse, an actuator interruption, or a change in the safe set.


{\color{black}
\begin{problem}
    Given system \eqref{eq:uncertain-system}, a compact set $\SP\subset\mathbb R^p$ such that $\theta^\star\in\SP$, a safe set $\mathcal C$, and a goal point $x_\rmc$,
    design a control policy $k:\mathbb R\times\mathcal D\times\mathbb R^p\to\mathcal U$ such that the closed-loop trajectories under $u=k(t,x,\hat\theta)$ satisfy
    \begin{enumerate}
        \item \textbf{Reach}: For all $x(0)\notin \mathcal C$, there exists $T<\infty$ such that $x(t)\in \mathcal C$ for all $t\geq T$, and for all $x(0)\in \mathcal C$, $x(t)\in\mathcal C$ for all $t\geq 0$; and
        \item \textbf{Stabilize}: For all $x(0)\in \mathcal D$, $\lim_{t\to \infty}x(t) = x_\rmc$. 
    \end{enumerate}
\end{problem}
}


In this paper, we construct a control law in which an estimate $\hat\theta(t)$ of $\theta^\star$ and a bound $\beta(t)$ on the estimation error $\tilde\theta\isdef\theta^\star-\hat\theta$ jointly certify safe-set recovery and convergence to the goal point.
Our proposed method ensures that the bound $\beta$ is computable online and nonincreasing, and satisfies $\|\tilde\theta(t)\|\leq\beta(t)$ for all $t\geq 0.$

The construction has four steps.
First, $\hat\theta$ is generated by a gradient update law driven by a filtered regressor.
Second, $\beta$ is constructed from a known bound on the initial estimation error, and is tightened by an accumulated regressor when the regressor is excited over an interval, without requiring persistence of excitation.
Third, a margin proportional to $\beta$ is added to the control barrier function condition and to the control Lyapunov function condition, each evaluated at $\hat\theta.$
Fourth, the two modified conditions are enforced in a single quadratic program.
The margins absorb the estimation error, and thus the finite-time recovery and convergence guarantees of the known-parameter design hold for the unknown parameter whenever the quadratic program is feasible.

\section{Parameter Estimation with Computable Error Bounds}
\label{sec:para_estimator}
This section constructs the estimate $\hat\theta$ and a bound $\beta$ such that $|\tilde\theta(t)|\leq\beta(t)$ for all $t\geq0.$
The uncertain dynamics \eqref{eq:uncertain-system} are written as a linear regressor equation, which is filtered to remove the dependence on $\dot x,$ and $\hat\theta$ is updated by a gradient law.
A bound on $\tilde\theta$ is then obtained by propagating an assumed bound on the initial error, and this bound is computable and nonincreasing.
Since the conservatism of this bound is fixed at initialization, an accumulated regressor is used to recover the estimation error exactly under interval excitation, which yields a tighter bound.

\subsection{Gradient-based parameter estimator}

To construct the estimator, note that \eqref{eq:uncertain-system} implies the
regressor equation
\begin{align}
    \dot x-f(x)-g(x)u=\phi(x)\theta^\star.
    \label{eq:regressor-eqn}
\end{align}
Since $x$ and $u$ are measured, all signals on the left-hand side of \eqref{eq:regressor-eqn} except
$\dot x$ are computable.
Since $\dot x$ is not measured, \eqref{eq:regressor-eqn} is filtered with
a strictly proper filter $R(s)$ to obtain the linear regressor
\begin{align}
    x_\rmf=\Phi_\rmf\theta^\star,
    \label{eq:linear-regression}
\end{align}
where
\begin{align}
    x_\rmf
        &\isdef
            R(s)[\dot x-f(x)-g(x)u],
    \label{eq:xf-def}\\
    \Phi_\rmf
        &\isdef
            R(s)[\phi(x)].
    \label{eq:phif-def}
\end{align}
where $x_\rmf\in\BBR^n,$ $\Phi_\rmf\in\BBR^{n\times p},$ and $R(s)$ is applied elementwise.
In particular, letting $R(s)=\frac{1}{s+p_\rmf},$ where $p_\rmf>0,$ yields
\begin{align}
    x_\rmf
        &=
            x-\frac{p_\rmf x}{s+p_\rmf}-\frac{f(x)+g(x)u}{s+p_\rmf},
    \label{eq:xf-realization}\\
    \Phi_\rmf
        &=
            \frac{\phi(x)}{s+p_\rmf},
    \label{eq:phif-realization}
\end{align}
where $\frac{sx}{s+p_\rmf}=x-\frac{p_\rmf x}{s+p_\rmf}$ is used to
eliminate $\dot x,$ and thus $x_\rmf$ and $\Phi_\rmf$ are computed online from
$x$ and $u$ \cite{delgado2024adaptive}.

Since $\Phi_\rmf$ and $x_\rmf$ are computed online, $\theta^\star$ is the only
unknown in \eqref{eq:linear-regression}.
The prediction error is
\begin{align}
    e\isdef x_\rmf-\Phi_\rmf\hat\theta
    \label{eq:prediction-error}
\end{align}
and, by \eqref{eq:linear-regression}, satisfies $e=\Phi_\rmf\tilde\theta.$
The estimate is updated by the gradient law
\begin{align}
    \dot{\hat\theta}=\Gamma\Phi_\rmf^\rmT e,
    \label{eq:gradient-law}
\end{align}
where $\hat\theta\in\BBR^p$ and $\Gamma>0$ is a scalar gain.

\subsection{Computable bound on the estimation error}

Let $\hat\theta(0)\in\BBR^p$ be an initial estimate.
\begin{assumption}
\label{assum:initial-bound}
A constant $\beta_0>0$ is known such that
$\|\theta^\star-\hat\theta(0)\|\le\beta_0.$
\end{assumption}

\begin{proposition}
\label{prop:error-bound}
Let Assumption \ref{assum:initial-bound} hold, and define
\begin{align}
    \beta_\rmp(t)
    \isdef
    \left(
    \beta_0^2-2\Gamma\int_{0}^{t}\|e(\tau)\|^2\,\rmd\tau
    \right)^{1/2}.
    \label{eq:beta-p}
\end{align}
Then, for all $t\ge0,$ $\beta_\rmp(t)$ is well defined and nonnegative,
$\beta_\rmp$ is nonincreasing, and $|\tilde\theta(t)|\le\beta_\rmp(t).$
\end{proposition}

\begin{proof}
Since $\dot{\tilde\theta}=-\dot{\hat\theta},$ it follows from
\eqref{eq:prediction-error} and \eqref{eq:gradient-law} that
\begin{align}
    \frac{\rmd}{\rmd t}\frac{\|\tilde\theta\|^2}{2}
    =
    -\Gamma\tilde\theta^\rmT\Phi_\rmf^\rmT\Phi_\rmf\tilde\theta
    =
    -\Gamma\|e\|^2,
    \label{eq:lyapunov-identity}
\end{align}
and thus, for all $t\ge0,$
\begin{align}
    \|\tilde\theta(t)\|^2
    =
    \|\tilde\theta(0)\|^2-2\Gamma\int_{0}^{t}\|e(\tau)\|^2\,\rmd\tau.
    \label{eq:error-identity}
\end{align}
Since $\|\tilde\theta(t)\|^2\ge0,$ \eqref{eq:error-identity} implies
$2\Gamma\int_{0}^{t}\|e(\tau)\|^2\,\rmd\tau\le\|\tilde\theta(0)\|^2,$ which, with
Assumption \ref{assum:initial-bound}, implies that the argument of the square
root in \eqref{eq:beta-p} is nonnegative, and hence $\beta_\rmp$ is well
defined and nonnegative.
Furthermore, $\beta_\rmp$ is nonincreasing since the integral in
\eqref{eq:beta-p} is nondecreasing.
Finally, \eqref{eq:error-identity} and Assumption \ref{assum:initial-bound}
imply $\|\tilde\theta(t)\|^2\le\beta_\rmp^2(t).$
\end{proof}

Since $e$ is computed online and $\Gamma$ is known, $\beta_\rmp$ is
computable.
In particular, $\beta_\rmp$ can be computed using the single integrator
\begin{align}
    \dot b=-2\Gamma\|e\|^2,
    \label{eq:beta-ode}
\end{align}
and $\beta_\rmp=\sqrt b,$ where $b(0)=\beta_0^2.$
Furthermore, since \eqref{eq:beta-ode} does not require inverting $\Phi_\rmf,$
$\beta_\rmp$ remains well defined at states where $\phi$ is unexcited.

\subsection{Exact estimation error under interval excitation}

Note that \eqref{eq:error-identity} implies
$\beta_\rmp^2(t)-\tilde\theta^2(t)=\beta_0^2-\tilde\theta^2(0),$ and thus the
conservatism of $\beta_\rmp$ is fixed at initialization and is not reduced by
data.
Estimators that avoid persistence of excitation include concurrent learning \cite{chowdhary2010concurrent} and dynamic regressor extension and mixing \cite{aranovskiy2016performance}.
To remove this floor, consider the accumulated regressor signals
$\overline\Phi$ and $\overline X$ defined by
\begin{align}
    \dot{\overline\Phi}
    &=
    -\mu\overline\Phi+\Phi_\rmf^\rmT\Phi_\rmf,
    &
    \overline\Phi(0)&=0,
    \label{eq:acc-phi}\\
    \dot{\overline X}
    &=
    -\mu\overline X+\Phi_\rmf^\rmT x_\rmf,
    &
    \overline X(0)&=0,
    \label{eq:acc-x}
\end{align}
where $\mu>0$ is the forgetting factor, and define
\begin{align}
    \Xi\isdef\overline\Phi\hat\theta-\overline X.
    \label{eq:xi}
\end{align}

\begin{proposition}
\label{prop:exact-error}
For all $t\ge0,$ $\overline\Phi(t)\ge0$ and
\begin{align}
    \Xi(t)=-\overline\Phi(t)\tilde\theta(t).
    \label{eq:xi-identity}
\end{align}
Furthermore, if $\overline\Phi(t)>0,$ then
$\tilde\theta(t)=-\overline\Phi^{-1}(t)\Xi(t).$
\end{proposition}

\begin{proof}
It follows from \eqref{eq:linear-regression}, \eqref{eq:acc-phi}, and
\eqref{eq:acc-x} that
\begin{align}
    \frac{\rmd}{\rmd t}
    \left(\overline X-\overline\Phi\theta^\star\right)
    =
    -\mu\left(\overline X-\overline\Phi\theta^\star\right),
    \label{eq:acc-error-dynamics}
\end{align}
which, since $\overline X(0)-\overline\Phi(0)\theta^\star=0,$ implies
$\overline X=\overline\Phi\theta^\star,$ and thus \eqref{eq:xi} implies
\eqref{eq:xi-identity}.
Furthermore, \eqref{eq:acc-phi} implies
\begin{align}
    \overline\Phi(t)
    =
    \int_0^te^{-\mu(t-\tau)}\|\Phi_\rmf(\tau)\|^2\,\rmd\tau,
    \label{eq:acc-phi-solution}
\end{align}
which is nonnegative.
\end{proof}

Since \eqref{eq:error-identity} implies that $|\tilde\theta|$ is
nonincreasing, every bound on $|\tilde\theta|$ obtained at time $\tau$ remains
valid for all $t\ge\tau.$
Propositions \ref{prop:error-bound} and \ref{prop:exact-error} thus yield the
bound
\begin{align}
    \beta(t)
    \isdef
    \min
    \left\{
    \beta_\rmp(t),\,
    \inf_{\tau\in\ST_t}
    \left|\overline\Phi^{-1}(\tau)\Xi(\tau)\right|
    \right\},
    \label{eq:beta}
\end{align}
where
\begin{align}
    \ST_t\isdef\{\tau\in[0,t]\colon\overline\Phi(\tau)>0\}.
    \label{eq:excitation-times}
\end{align}

\begin{corollary}
\label{cor:beta}
For all $t\ge0,$ $\beta(t)\ge0,$ $\beta$ is nonincreasing, and
$|\tilde\theta(t)|\le\beta(t).$
\end{corollary}

\begin{proof}
Proposition \ref{prop:error-bound} implies $\beta_\rmp\ge0,$ and Proposition
\ref{prop:exact-error} implies that each term in the infimum in
\eqref{eq:beta} is nonnegative, and hence $\beta\ge0.$
Furthermore, $\beta_\rmp$ is nonincreasing, and $\ST_{t_1}\subseteq\ST_{t_2}$
for all $t_2\ge t_1,$ which implies that the infimum in \eqref{eq:beta} is
nonincreasing, and hence $\beta$ is nonincreasing.
Finally, Proposition \ref{prop:exact-error} implies that, for all
$\tau\in\ST_t,$ $|\tilde\theta(\tau)|=|\overline\Phi^{-1}(\tau)\Xi(\tau)|,$
which, since $|\tilde\theta|$ is nonincreasing by \eqref{eq:error-identity},
implies $|\tilde\theta(t)|\le|\overline\Phi^{-1}(\tau)\Xi(\tau)|,$ and
combining with Proposition \ref{prop:error-bound} yields
$|\tilde\theta(t)|\le\beta(t).$
\end{proof}

Note that the infimum in \eqref{eq:beta} is a running minimum and is thus computed by a monotone hold.
Furthermore, \eqref{eq:acc-phi-solution} implies that $\overline\Phi(t)\succ0$ if and only if no nonzero $v\in\BBR^p$ satisfies $\Phi_\rmf(\tau)v=0$ for almost every $\tau\in[0,t],$ and thus $\ST_t$ is nonempty under interval excitation in every parameter direction rather than persistent excitation.

\section{Control design with parameter adaptation}
\label{sec:control}
This section uses $\hat\theta$ and $\beta$ from Section III to enforce safety and convergence under the unknown parameter $\theta^\star.$
The CBF and CLF conditions of the known-parameter design contain $\theta^\star$ and thus cannot be evaluated.
Each condition is instead evaluated at $\hat\theta$ and augmented with a margin, $\delta_h$ for the CBF condition and $\delta_V$ for the CLF condition, that bounds the contribution of $\tilde\theta$ through $\beta.$
Theorem \ref{thm:safety with delta} shows that the modified CBF condition yields forward invariance of $\mathcal C$ and finite-time recovery to $\mathcal C,$ with a recovery-time bound that does not depend on $\theta^\star$ or $\beta.$
Theorem \ref{thm:convergence with delta} shows that the modified CLF condition yields exponential convergence of $V.$
The two modified conditions are then enforced in a single quadratic program.

\subsection{Modified CBF condition}
When the parameter $\theta^\star$ is known, the CBF condition reads
\begin{align}
    L_fh(x) + L_gh(x)u + L_\phi h(x)\theta^\star & \nonumber\\
    +\gamma\operatorname{sgn}(h(x))|h(x)|^\rho & \geq 0,
\end{align}
for all $x\in \mathcal C$.
However, since the parameter $\theta^\star$ is unknown, this condition cannot be enforced.
Using the estimate $\hat \theta$, we use the following modified condition
\begin{align}
    L_fh(x) + L_gh(x)u + L_\phi h(x)\hat \theta&\nonumber \\
    +\gamma\operatorname{sgn}(h(x))|h(x)|^\rho&\geq \delta_h,
\end{align}
where $\delta_h$ is an appropriately chosen estimation-error dependent safety margin.
Based on Corollary \ref{cor:beta}, we can choose $\delta_h(t) = \|L_\phi h(x)\||\beta(t)|$ at time $t\geq 0$ to ensure safety, as shown the following result.

\begin{theorem}\label{thm:safety with delta}
    Under the conditions of Corollary \ref{cor:beta}, let $\gamma>0$ and $\rho\in(0,1),$ and let
    \begin{align}\label{eq: delta CBF}
        \delta_h(t) = \|L_\phi h(x)\||\beta(t)|.
    \end{align}
    If there exists a control signal $u:\mathbb R\to \mathcal U$ such that
    \begin{align}\label{eq: cbf delta cond}
        L_fh(x(t)) + L_gh(x(t))u + L_\phi h(x)\hat \theta(t) & \nonumber\\
        +\gamma\operatorname{sgn}(h(x))|h(x)|^\rho & \geq \delta_h(t)
    \end{align}
    along the solution $x(\cdot)$ of \eqref{eq:uncertain-system} for all $t\geq 0$, then the set $\mathcal C$ is forward invariant for \eqref{eq:uncertain-system} for every $x(0)\in \mathcal C.$
    Furthermore, for every $x(0)\notin \mathcal C,$ the trajectory reaches $\mathcal C$ at a time
    \begin{align}\label{eq: recovery time}
        T\le\frac{|h(x(0))|^{1-\rho}}{\gamma(1-\rho)}.
    \end{align}
\end{theorem}
\begin{proof}
    Under the conditions of Corollary \ref{cor:beta}, it holds that the parameter estimation error is bounded by $\beta$.
    For a control signal $u$, let $x$ be the solution of \eqref{eq:uncertain-system}.
    If the condition \eqref{eq: cbf delta cond} holds along the solution $x(\cdot)$, we have
    \begin{align*}
        L_fh(x(t)) + L_gh(x(t))u + & L_\phi h(x)\hat \theta(t) \nonumber\\
        +\gamma\operatorname{sgn}(h(x))|h(x)|^\rho & \geq \delta_h(t)\\
        \implies L_fh(x(t)) + L_gh(x(t))u + & L_\phi h(x)\theta^\star  \nonumber\\
        +\gamma\operatorname{sgn}(h(x))|h(x)|^\rho & \geq \delta_h(t) + L_\phi h(x)\tilde \theta(t).
    \end{align*}
    Now, it holds that
    \begin{align*}
        \delta_h(t) + L_\phi h(x)\tilde \theta(t) & \geq \delta_h(t) - \|L_\phi h(x)\tilde \theta(t)\|\\
        & \geq \delta_h(t) - \|L_\phi h(x)\||\beta(t)|\\
        & \geq 0,
    \end{align*}
    for all $t\geq 0$, where the second inequality follows from Corollary \ref{cor:beta} and the last inequality follows from \eqref{eq: delta CBF}.
    Consequently, we have that
    \begin{align*}
        \dot h(t) \geq -\gamma\operatorname{sgn}(h(x(t)))|h(x(t))|^\rho.
    \end{align*}
    Suppose $x(0)\in\mathcal C$ and $h(x(t_1))<0$ for some $t_1>0.$
    Let $t_0\in[0,t_1)$ be the last time at which $h(x(t_0))\geq0.$
    Then $h(x(t_0))=0$ and $h(x(t))<0$ on $(t_0,t_1],$ so $\dot h\geq\gamma|h|^\rho>0$ on $(t_0,t_1],$ which contradicts $h(x(t_1))<h(x(t_0)).$
    Hence $\mathcal C$ is forward invariant.
    Now suppose $h(x(0))<0$ and let $w=-h.$
    While $w>0,$ $\dot w\leq-\gamma w^\rho,$ which implies
    \begin{align*}
        \frac{\rmd}{\rmd t}w^{1-\rho}=(1-\rho)w^{-\rho}\dot w\leq-\gamma(1-\rho).
    \end{align*}
    Since $w^{1-\rho}\geq0,$ $w$ reaches zero no later than the time in \eqref{eq: recovery time}, and forward invariance then holds from that time on.
\end{proof}

\subsection{Modified CLF condition}
Let $V\colon\mathcal D\to\BBR$ be a $C^1$ control Lyapunov function for the goal point $x_\rmc,$ that is, $V(x_\rmc)=0$ and $\alpha(\|x-x_\rmc\|)\leq V(x)$ for all $x\in\mathcal D$ and some class-$\mathcal K$ function $\alpha.$
When $\theta^\star$ is known, the CLF condition reads
\begin{align}
    L_fV(x) + L_gV(x)u + L_\phi V(x)\theta^\star \leq -\lambda V(x)
\end{align}
for some $\lambda>0,$ and implies $\dot V\leq-\lambda V.$
Since $\theta^\star$ is unknown, this condition cannot be enforced.
The modified condition evaluates the left-hand side at $\hat\theta$ and adds a margin $\delta_V,$ that is,
\begin{align}
    L_fV(x) + L_gV(x)u + L_\phi V(x)\hat \theta \leq -\lambda V(x) - \delta_V,
\end{align}
where $\delta_V$ bounds the contribution of $\tilde\theta$ to $\dot V.$
By Corollary \ref{cor:beta}, the choice $\delta_V(t) = \|L_\phi V(x)\||\beta(t)|$ at time $t\geq 0$ ensures convergence, as shown in the following result.

\begin{theorem}\label{thm:convergence with delta}
    Under the conditions of Corollary \ref{cor:beta}, with
    \begin{align}\label{eq: delta CLF}
        \delta_V(t) = \|L_\phi V(x)\||\beta(t)|,
    \end{align}
    if there exists a control signal $u:\mathbb R\to \mathcal U$ such that the following holds:
    \begin{align}\label{eq: clf delta cond}
         L_fV(x(t)) + L_gV(x(t))u + L_\phi V(x(t))\hat \theta(t) \nonumber \\
         \leq -\lambda V(x(t)) - \delta_V(t),
    \end{align}
    along the solution $x(\cdot)$ of \eqref{eq:uncertain-system} for every $x(0)\in \mathcal D$ and for all $t\geq 0$, then $V(x(t))$ converges to $0$ exponentially and $x(t)$ converges to $x_\rmc$ for every $x(0)\in \mathcal D$.
\end{theorem}
\begin{proof}
    Under the conditions of Corollary \ref{cor:beta}, it holds that the parameter estimation error is bounded by $\beta$.
    For a control signal $u$, let $x$ be the solution of \eqref{eq:uncertain-system}.
    If the condition \eqref{eq: clf delta cond} holds along the solution $x(\cdot)$, we have
    \begin{align*}
       L_fV(x(t)) + & L_gV(x(t))u + L_\phi V(x(t))\hat \theta(t) \\
       \leq &-\lambda V(x(t)) - \delta_V(t)\\
       \implies L_fV(x(t)) + & L_gV(x(t))u + L_\phi V(x(t))\theta^\star \\
       \leq &-\lambda V(x(t)) - \delta_V(t) + L_\phi V(x(t))\tilde \theta(t).
    \end{align*}
    Now, it holds that
    \begin{align*}
        L_\phi V(x(t))\tilde \theta(t) - \delta_V(t) & \leq \|L_\phi V(x(t))\tilde \theta(t)\| - \delta_V(t)\\
         & \leq \|L_\phi V(x(t))\||\beta(t)| - \delta_V(t)\\
         & \leq 0,
    \end{align*}
    where the second inequality follows from Corollary \ref{cor:beta} and the last inequality follows from \eqref{eq: delta CLF}.
    Consequently, we have that
    \begin{align*}
     L_fV(x(t)) + & L_gV(x(t))u + L_\phi V(x(t))\theta^\star \\
       \leq &-\lambda V(x(t)),
    \end{align*}
    which implies $\dot V(x(t))\leq-\lambda V(x(t)),$ and thus $V(x(t))\leq e^{-\lambda t}V(x(0))$ for all $t\geq0.$
    Since $\alpha(\|x(t)-x_\rmc\|)\leq V(x(t))$ and $\alpha$ is of class $\mathcal K,$ $x(t)$ converges to $x_\rmc.$
\end{proof}

\subsection{Control design using quadratic programming}
The modified CBF and CLF conditions are affine in $u,$ and $\mathcal U$ is a polytope.
The control input is thus obtained at each time $t$ from the quadratic program
\begin{subequations}\label{eq:delta QP}
\begin{align}
    \min_{u\in \mathcal U} \quad & u^\rmT u \\
    \text{s.t.} \quad & L_fh(x) + L_gh(x)u + L_\phi h(x)\hat \theta \nonumber\\
    & \qquad +\gamma\operatorname{sgn}(h(x))|h(x)|^\rho \geq \delta_h, \\
    & L_fV(x) + L_gV(x)u + L_\phi V(x)\hat \theta \leq -\lambda V(x) - \delta_V,
\end{align}
\end{subequations}
where $x=x(t),$ $\hat\theta=\hat\theta(t),$ and $\delta_h$ and $\delta_V$ are given by \eqref{eq: delta CBF} and \eqref{eq: delta CLF}, respectively.
Both constraints are enforced as hard constraints, and the QP is not guaranteed to be feasible.
Feasibility is thus a hypothesis of the following result, which combines Theorems \ref{thm:safety with delta} and \ref{thm:convergence with delta}.

\begin{remark}
\label{rem:feasibility}
The set $\mathcal U$ enters Theorems \ref{thm:safety with delta} and \ref{thm:convergence with delta} only through the existence of a feasible $u,$ and the proofs hold for $\mathcal U=\BBR^m.$
The bound on $u$ is one cause of infeasibility, since the recovery term $\gamma|h|^\rho$ and the margins $\delta_h$ and $\delta_V,$ which are largest at initialization where $\beta=\beta_0,$ increase the required control authority.
Enlarging $\mathcal U$ thus enlarges the set of states at which the QP is feasible.
Without a bound on $u,$ the QP is infeasible at a state $x$ only if $L_gh(x)=0$ or $L_gV(x)=0$ and the drift terms do not satisfy the corresponding constraint, or if $L_gh(x)$ and $L_gV(x)$ are positively proportional and their right-hand sides are incompatible.
\end{remark}

\begin{corollary}
    Under the conditions of Corollary \ref{cor:beta}, if the QP \eqref{eq:delta QP} is feasible for all $x\in \mathcal D$ and the solution $u$ of the QP is continuous, then the following holds for the resulting closed-loop solution $x(\cdot)$ of \eqref{eq:uncertain-system}:
    \begin{enumerate}
        \item For all $x(0)\in \mathcal D,$ $V(x(t))$ converges to $0$ exponentially and $x(t)$ converges to $x_\rmc;$
        \item For all $x(0)\in \mathcal C,$ $x(t)\in \mathcal C$ for all $t\geq 0;$ and
        \item For all $x(0)\notin \mathcal C,$ there exists $T\leq |h(x(0))|^{1-\rho}/(\gamma(1-\rho))$ such that $x(t)\in \mathcal C$ for all $t\geq T.$
    \end{enumerate}
\end{corollary}
The bound on $T$ in item 3 depends on $x(0),$ $\gamma,$ and $\rho$ only, and thus does not depend on $\theta^\star$ or $\beta.$

\section{Numerical Simulation}
\label{sec:numerical}

This section applies the proposed controller to a constructed unstable system in which $f,$ $g,$ and $\phi$ are state dependent.
The adaptive controller is compared with an oracle controller that uses $\theta^\star$ and a naive controller that uses an incorrect constant value of the parameter.
One initial condition outside $\mathcal C$ tests finite-time recovery, and one inside $\mathcal C$ tests forward invariance.

\subsection{Setup}
\label{sec:num-setup}
Consider the nonlinear system of the form \eqref{eq:uncertain-system}, where
\begin{gather}
    f(x) = \begin{bmatrix} -x_2 \\ x_1 \end{bmatrix},
    \quad
    g(x) = (1+0.1x_1^2)I_2,
    \\
    \phi(x) = \begin{bmatrix} (1-x_1^2)x_2 \\ 0 \end{bmatrix},
    \label{eq:num-system}
\end{gather}
with $\theta^\star = 2.$
The objective is to stabilize $x$ to a goal point $x_\rmc\in\BBR^2,$ to ensure that $x$ reaches the safe set $\mathcal C$ in finite time, and to ensure that $\mathcal C$ remains forward invariant thereafter.
In this example, we assume that the safe set is $\mathcal C = \{x : h(x) \geq 0\}$ with $h(x) = r^2 - \|x\|^2$ and $r=1.5,$
and $\mathcal U = \{u\in\BBR^2 : \|u\|_\infty \leq 3\}.$
The CLF is $V(x) = \|x-x_\rmc\|^2$ for the goal point $x_\rmc.$

The choice of $x_\rmc$ affects both the estimator and the closed-loop equilibrium.
If $x_\rmc=0,$ then $x(t)$ converges to a state at which $\phi(x)=0,$ and thus the regressor $\Phi_\rmf \to 0$.
Consequently, $\beta$ stops contracting once the trajectory settles.
Choosing $x_\rmc$ such that $\phi(x_\rmc)\neq0$ instead keeps the regressor persistently nonzero at steady state and allows \eqref{eq:beta} to contract with data.
%

%
The estimator parameters are $p_\rmf = 5,$ $\Gamma = 5,$ $\mu = 0.5,$ and $\hat\theta(0)=0$ with $\beta_0 = 2 = |\theta^\star-\hat\theta(0)|,$ so Assumption \ref{assum:initial-bound} holds with equality.
The CBF/CLF parameters are $\gamma=1,$ $\rho=0.5,$ and $\lambda=1.$
The goal point is $x_\rmc = (0,1),$ at which $\phi(x_\rmc) = 1 \neq 0.$


We compare three controllers, each obtained from the QP \eqref{eq:delta QP} with a different choice of $\theta_\mathrm{used}$ and margins.
The \textit{oracle} controller uses $\theta_\mathrm{used}=\theta^\star$ with $\delta_h=\delta_V=0.$
The \textit{naive} controller uses a constant estimate $\theta_\mathrm{used}=\hat\theta(0)=0$ with $\delta_h=\delta_V=0,$ and thus lies outside the hypotheses of Theorems \ref{thm:safety with delta} and \ref{thm:convergence with delta}.
The \textit{adaptive} controller uses $\theta_\mathrm{used}=\hat\theta(t)$ together with $\delta_h(t) = |L_\phi h(x)|\beta(t)$ and $\delta_V(t) = |L_\phi V(x)|\beta(t)$ as in \eqref{eq: delta CBF} and \eqref{eq: delta CLF}.
The QP is feasible whenever $g(x)$ retains rank two on the region visited by the trajectory, since the CBF and CLF rows are then two generically compatible half-planes in $\BBR^2$ rather than two competing bounds on a single scalar input.

In all figures, blue denotes the oracle controller, red the naive controller, and yellow the adaptive controller.
In the diagnostic figures, $u_1$ is shown solid and $u_2$ dashed.

\subsection{Finite-Time Recovery}
The initial condition is $x(0) = (1.8,-0.6),$ outside the safe set, with $h(x(0)) = -1.35.$
Figure \ref{fig:fig_recovery_diagnostics} shows $h,$ $V,$ the bound $\beta,$ the parameter error $|\tilde \theta|$ for the adaptive controller, and the input $u.$
Figure \ref{fig:fig_phase_portrait} shows the phase-plane trajectories.


\begin{figure}[ht]
\centering
\includegraphics[width=\linewidth]{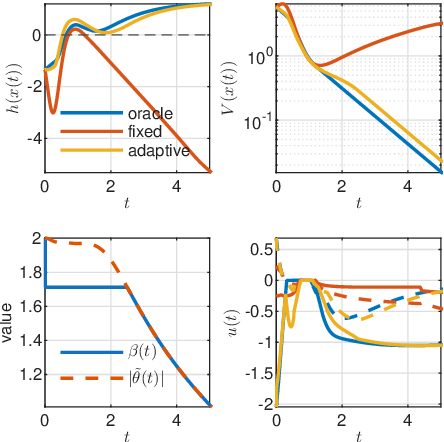}
\caption{Barrier value $h$, Lyapunov value ($V$ on a log scale), estimator diagnostics ($\beta(t)$ and $|\tilde \theta|$), and control input $u.$.}
\label{fig:fig_recovery_diagnostics}
\end{figure}

\begin{figure}[ht]
\centering
\includegraphics[width=\linewidth]{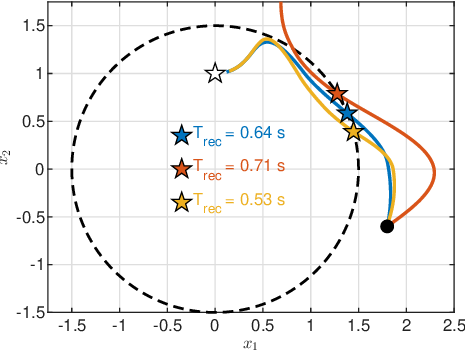}
\caption{Phase-plane trajectories, with the safe set boundary, initial condition $x(0),$ CLF goal $x_\rmc,$ and recovery points marked with a star; the recovery time at each marker is given in matching color.}
\label{fig:fig_phase_portrait}
\end{figure}


\begin{table}[ht]
\centering
\caption{Recovery time, recovery-time bound, and terminal values, $T=5\,\mathrm s.$}
\label{tab:num-example}
\begin{tabular}{lcccc}
\hline
 & $T_\mathrm{rec}$ (s) & $T_\mathrm{bound}$ (s) & $h(T)$ & $V(T)$ \\
\hline
oracle & 0.636 & 2.324 & 1.202 & 0.0153 \\
naive, no margin & 0.707$^\ast$ & --- & $-5.320$ & 3.2385 \\
adaptive, margin & 0.528 & 2.324 & 1.178 & 0.0228 \\
\hline
\end{tabular}
\\[2pt]
\footnotesize{$T_\mathrm{bound}$ is given by \eqref{eq: recovery time} and does not apply to the naive controller.
$^\ast$Crosses $h=0$ transiently; does not remain in $\mathcal C.$}
\end{table}

Table \ref{tab:num-example} summarizes the results.
The oracle and adaptive controllers enter $\mathcal C$ below the bound \eqref{eq: recovery time} and satisfy $h(t)\geq0$ afterwards, and both drive $V$ toward zero.
The naive controller crosses $h=0$ once and then leaves $\mathcal C,$ and $V$ plateaus near $3.2$ instead of converging to zero.
This is consistent with the naive controller not satisfying the modified condition \eqref{eq: cbf delta cond}.

For the adaptive controller, $\beta$ is nonincreasing throughout, and the accumulated-regressor term in \eqref{eq:beta} reduces $\beta$ below $\beta_\rmp.$
%
%
The input satisfies $\|u\|_\infty\leq3$ for all three controllers.

\subsection{Forward Invariance}

The initial condition is $x(0)=(0.5,-1.2),$ inside the safe set, with $h(x(0))=0.560.$
All other parameters are unchanged.
Figures \ref{fig:fig_recovery_diagnostics_FI} and \ref{fig:fig_phase_portrait_FI} show the same quantities as in the previous subsection.

\begin{figure}[ht]
\centering
\includegraphics[width=\linewidth]{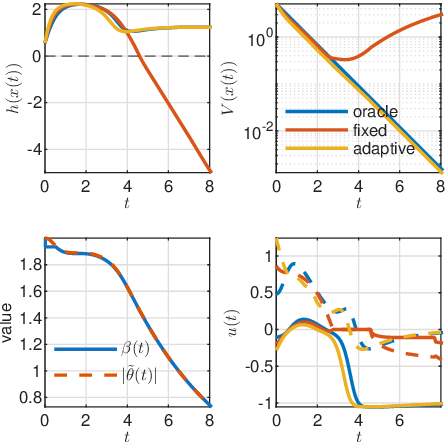}
\caption{Barrier value $h$, Lyapunov value ($V$ on a log scale), estimator diagnostics ($\beta(t)$ and $|\tilde \theta|$), and control input $u.$}
\label{fig:fig_recovery_diagnostics_FI}
\end{figure}

\begin{figure}[ht]
\centering
\includegraphics[width=\linewidth]{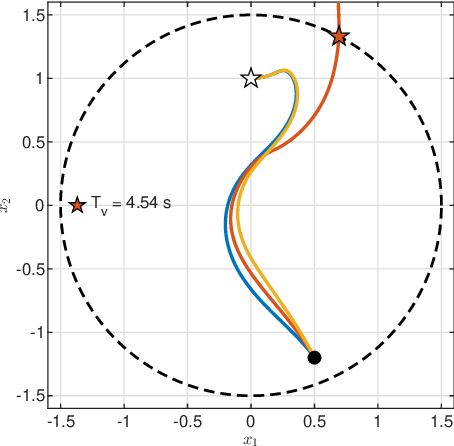}
\caption{Phase-plane trajectories for $x(0)=(0.5,-1.2)$ with the safe set boundary, initial condition $x(0),$ and goal $x_\rmc.$
The star marks the point at which the naive controller leaves $\mathcal C,$ and the violation time is given in matching color.}

\label{fig:fig_phase_portrait_FI}
\end{figure}


\begin{table}[ht]
\centering
\caption{Violation time and terminal values, $T=8\,\mathrm s,$ $x(0)=(0.5,-1.2).$}
\label{tab:num-example-fwdinv}
\begin{tabular}{lccc}
\hline
 & $T_\rmv$ (s) & $h(T)$ & $V(T)$ \\
\hline
oracle & --- & 1.245 & 0.0017 \\
naive, no margin & 4.54 & $-4.999$ & 3.0402 \\
adaptive, margin & --- & 1.246 & 0.0013 \\
\hline
\end{tabular}
\end{table}

Table \ref{tab:num-example-fwdinv} summarizes the results.
The oracle and adaptive controllers remain in $\mathcal C$ over the horizon and drive $V$ toward zero, consistent with Theorems \ref{thm:safety with delta} and \ref{thm:convergence with delta}.
The naive controller leaves $\mathcal C$ at $t=4.54\,\mathrm s$ although $x(0)\in\mathcal C.$
Omitting the margin thus costs forward invariance, in addition to finite-time recovery.
%


\section{Conclusions}
\label{sec:conclusions}

This paper presented a controller that solves the reach-stabilize problem for control-affine systems with an unknown constant parameter vector multiplying a known regressor matrix.
A gradient estimator with a filtered regressor generates an estimate $\hat\theta,$ and a bound $\beta$ on the estimation error is computed online from a known bound on the initial error.
The bound is nonincreasing, and an accumulated regressor tightens it under interval excitation without requiring persistence of excitation.
Margins proportional to $\beta$ are added to the CBF and CLF conditions, and both conditions are enforced in a single QP under the same estimate.
The closed loop remains in the safe set from an initial state inside it, enters the safe set in finite time from an initial state outside it, and converges to the goal point.
The recovery-time bound equals that of the known-parameter design and does not depend on the unknown parameter.
In the numerical examples, the controller that ignores the estimation error leaves the safe set from initial conditions both outside and inside it, while the proposed controller remains in the set.

Future work is focused on time-varying safe sets.


\bibliographystyle{ieeetr}
\bibliography{bibliography}

\end{document}